\documentclass[]{amsart}

\usepackage{graphicx,color,hyperref,verbatim}
\numberwithin{equation}{section}

\usepackage[initials,nobysame]{amsrefs}
\usepackage{mathrsfs}

\usepackage{enumerate}

\newtheorem{theorem}{Theorem}[section]
\newtheorem*{theorem*}{Theorem}

\newtheorem{corollary}[theorem]{Corollary}
\newtheorem{lemma}[theorem]{Lemma}
\theoremstyle{definition}
\newtheorem{definition}[theorem]{Definition}

\theoremstyle{remark}
\newtheorem{remark}[theorem]{Remark}

\DeclareMathOperator{\sech}{sech}
\DeclareMathOperator{\cn}{cn}
\DeclareMathOperator{\dn}{dn}
\DeclareMathOperator{\sn}{sn}
\DeclareMathOperator{\am}{am}
\DeclareMathOperator{\sgn}{sign}

\newcommand{\R}{\mathbf{R}}
\newcommand{\N}{\mathbf{N}}
\newcommand{\Z}{\mathbf{Z}}
\newcommand{\A}{\mathcal{A}}

\newcommand{\SolnSpace}{W_\mathrm{arc}^{2,2}(0,L;\R^2)}
\newcommand{\SA}{\mathcal{A}}
\newcommand{\SB}{\mathcal{B}}
\newcommand{\SC}{\mathcal{C}}
\newcommand{\SF}{\mathcal{F}}
\newcommand{\SM}{\mathcal{M}}

\newcommand{\ef}[1]{\ensuremath{\mathsf{#1}}}

\begin{document}

\title[Uniqueness and minimality of Euler's elastica]{Uniqueness and minimality of Euler's elastica with monotone curvature}

\author{Tatsuya Miura}
\address[T.~Miura]{Department of Mathematics, Tokyo Institute of Technology, 2-12-1 Ookayama, Meguro-ku, Tokyo 152-8551, Japan; Present address: Department of Mathematics, Graduate School of Science, Kyoto University, Kitashirakawa Oiwake-cho, Sakyo-ku, Kyoto 606-8502, Japan}
\email{tatsuya.miura@math.kyoto-u.ac.jp}

\author{Glen Wheeler}
\address[G.~Wheeler]{School of Mathematics and Applied Statistics, University of Wollongong, Wollongong NSW 2522 Australia}
\email{glenw@uow.edu.au}

\date{\today}

\begin{abstract}
    For an old problem of Euler's elastica we prove the novel global property that every planar elastica with non-constant monotone curvature is uniquely minimal subject to the clamped boundary condition.
    We also partly extend this unique minimality to the length-penalised case; this result is new even in view of local minimality.
    As an application we prove uniqueness of global minimisers in the straightening problem for generic boundary angles.
\end{abstract}

\maketitle

\section{Introduction}

Euler's elastica is the oldest variational model of elastic rods initiated by D.\ Bernoulli and L.\ Euler in the 18th century.
The problem is to study critical points of the bending energy
$$\ef{B}[\gamma]=\int_\gamma k^2ds$$
among immersed planar curves $\gamma$ under the fixed-length constraint.
Here we use $k$ to denote the signed curvature of $\gamma$, and use $s$ and $ds$ for the arc-length parameter and measure, respectively.
By the multiplier method, such a critical point satisfies for some $\lambda\in\R$ the Euler--Lagrange equation of the modified bending energy
$$\ef{E}^\lambda[\gamma]
:= \ef{B}[\gamma] + \lambda \ef{L}[\gamma]
= \int_\gamma (k^2+\lambda)\, ds,$$
where $\ef{L}$ denotes the length functional,
which takes the form
\begin{equation}
\label{EQelastica}
    2k_{ss}+k^3-\lambda k=0\,.
\end{equation}
An arc-length parametrised curve $\gamma:[0,L]\to\R^2$ is called an \emph{elastica} (or \emph{$\lambda$-elastica} to explicitly specify the multiplier $\lambda$) if the curvature satisfies \eqref{EQelastica}.
As one may imagine due to their classicality, elasticae have been the subject of a great number of works.
However, compared to the level of critical points as well as local variational properties (stability), global variational properties of those critical points have retained their mystery and are relatively poorly understood.

In this paper we address global uniqueness and minimality issues subject to the classical \emph{clamped boundary condition}, which means that the endpoints are prescribed up to first order, exactly as posed by Bernoulli.
For closed curves it follows by Cauchy--Schwarz and Fenchel inequalities that minimisers are circles (the area-constrained counterpart is however much more nontrivial, see recent works \cite{BucurHenrot2017,FeroneKawohlNitsch2016}).
In the general (non-closed) case, global minimisers always exist, but their properties are largely unknown, including their uniqueness.

Here we adopt the natural energy space of Sobolev class $W^{2,2}$.
For $L>0$ we define the set of arc-length parametrised curves of length $L$ by
$$W_\mathrm{arc}^{2,2}(0,L;\R^2):=\Big\{ \gamma \in W^{2,2}(0,L;\R^2) \,:\, |\gamma'|\equiv1 \Big\}.$$
When we say that an arc-length parametrised curve $\gamma:[0,L]\rightarrow\R^2$ is an elastica, we mean that $\gamma\in\SolnSpace$ and the curvature scalar $k$ of $\gamma$ satisfies \eqref{EQelastica} in the distributional sense (which automatically implies smoothness).
In addition, given $\gamma\in W_\mathrm{arc}^{2,2}(0,L;\R^2)$ and $\tilde{L}>0$, let $\mathcal{A}(\gamma;\tilde{L})$ denote the set of curves that have length $
\tilde{L}$ and agree with $\gamma$ at the endpoints up to first order, that is,
$$\mathcal{A}(\gamma;\tilde{L}):=\Big\{ \tilde{\gamma} \in W_\mathrm{arc}^{2,2}(0,\tilde{L};\R^2) \,:\, \text{$\tilde{\gamma}$ satisfies \eqref{eq:clamped_gamma}} \Big\},$$
where \eqref{eq:clamped_gamma} are the following clamped boundary conditions:
\begin{equation}\label{eq:clamped_gamma}
    \text{$\tilde{\gamma}(0)=\gamma(0)$, $\tilde{\gamma}(\tilde{L})=\gamma(L)$, $\tilde{\gamma}_s(0)=\gamma_s(0)$, $\tilde{\gamma}_s(\tilde{L})=\gamma_s(L)$\,.}
\end{equation}
In particular, if $\tilde{L}=L$, we simply write 
$$\mathcal{A}(\gamma):=\mathcal{A}(\gamma;L).$$

Our first main result is the following sufficient condition on unique minimality.

\begin{theorem}\label{thm:main_fixed_length}
    Let $\gamma:[0,L]\to\R^2$ be an arc-length parametrised elastica with non-constant and monotone signed curvature.
    Then $\gamma$ is the unique minimiser of the bending energy $\ef{B}$ in the set $\mathcal{A}(\gamma)$.
\end{theorem}

Recall that in general an elastica may not be minimal, and even if minimal, may not be unique; see e.g.\ \cite[Figure 5]{miu20} for a typical scenario.
The case of constant curvature is easier but needs to be addressed independently.
A straight segment is always uniquely minimal.
A circular arc is uniquely minimal if and only if it is less than one fold; see \cite[Appendix B]{miu20} for minimality and \cite[Theorem 2.1]{MiuraYoshizawa2024crelle} for non-minimality.
In particular, the exactly one-fold circle is minimal but not unique by reflection; Theorem \ref{thm:main_fixed_length} also rules out such symmetry-induced non-uniqueness.

Few results are previously known about sufficient conditions for global minimality.
One instance is the following fact claimed by Sachkov in \cite{sac08a}:
for any elastica $\gamma$ there is a short subarc $\gamma|_{[0,s_0]}$ which is minimal.
Our theorem gives an entirely new sufficient condition, without any implicit smallness assumption.
Not only that, our result (combined with the constant-curvature case) can be regarded as an improvement of Sachkov's claim since the subarc always has monotone curvature $k|_{[0,s_0]}$ for small $s_0$.
We also mention that Sachkov--Sachkova obtained a general restriction on the number of global minimisers \cite{SS14}.
Concerning necessary conditions, see \cite{MiuraYoshizawa2024crelle} and the references therein.

As for stability (local minimality) there are several known results.
Among them, by Born's pioneering result \cite{born06} it is classically known that if an elastica is locally strictly convex, then it is stable.
More recently, Sachkov \cite{sac08a,sac08b} developed a precise stability theory by an optimal control approach.
In particular, Sachkov's results imply that if an elastica has monotone curvature, then it is also stable.
Our theorem reveals, unexpectedly to the authors, that Sachkov's local principle directly extends to its global counterpart (up to the case of constant curvature).
Note that such a global extension does not hold for Born's principle since long orbitlike elasticae are always non-minimal.
See also \cite{Maddocks1984,LangerSinger_85} for related stability analysis of planar or spatial elasticae subject to various boundary conditions.

Our next result partly extends Theorem \ref{thm:main_fixed_length} from the length-fixed to the length-penalised problem.
For $\gamma\in W_\mathrm{arc}^{2,2}(0,L;\R^2)$, we define the set of all curves which agree with $\gamma$ at the endpoints up to first order but may have different length by
$$\widetilde{\mathcal{A}}(\gamma):=\bigcup_{\tilde{L}>0}\mathcal{A}(\gamma;\tilde{L}).$$

\begin{theorem}\label{thm:main_length_variable}
    Let $\lambda\in\R$ and $\gamma:[0,L]\to\R^2$ be an arc-length parametrised $\lambda$-elastica with non-constant and monotone signed curvature.
    If $\gamma$ is not a subcritical wavelike elastica, then $\gamma$ is the unique minimiser of the modified bending energy $\ef{E}^\lambda$ in the set $\widetilde{\mathcal{A}}(\gamma)$.
\end{theorem}

\begin{remark}
    Wavelike elastica, each determined by an elliptic parameter $m\in(0,1)$, can be either \emph{supercritical}, \emph{subcritical}, or \emph{critical}.
    The latter case is that of the figure-eight elastica with $m=m_0$, whereas the remaining two have $m$ larger or smaller than that of the figure-eight elastica.
    See Definition \ref{def:supercritical} for details.
\end{remark}

To the authors' knowledge this is the first result about global minimality in the length-penalised problem.
Note that length-penalised minimality does not follow by length-fixed minimality, while the converse does follow.
Indeed, Theorem \ref{thm:main_length_variable} requires the additional ``non-subcriticality'' assumption.
Although its optimality is not clear, if we remove it then there are counterexamples (see Remark \ref{rem:multiplier}).

In addition, Theorem \ref{thm:main_length_variable} would be new even in view of stability.
The authors are only aware of the fact that Born's argument in \cite{born06} directly works even in the length-penalised case (but only for smooth perturbations), see also \cite[Section 5.3]{miu20}.
Most of the other known stability results are based on the fixed-length constraint.

Finally, we apply Theorem \ref{thm:main_length_variable} to directly solve Euler's elastica problem.
More precisely, for the classical clamped boundary value problem, we aim at detecting global minimisers \emph{only} from given boundary data.
Given $\ell\in[0,L)$ and $\theta_0,\theta_1\in(-\pi,\pi]$, we define the (nonempty) admissible set
\begin{equation}\label{eq:admissible_clamped}
    \A^L_{\ell,\theta_0,\theta_1} := \Big\{ \gamma\in W_\mathrm{arc}^{2,2}(0,L;\R^2) \;:\; \text{$\gamma$ satisfies \eqref{eq:clamped_BC}} \Big\},
\end{equation}
where the clamped boundary condition is given by
\begin{equation}\label{eq:clamped_BC}
    \gamma(0)=
\begin{pmatrix}
    0\\ 0
\end{pmatrix},\ 
\gamma(L)=
\begin{pmatrix}
    \ell\\ 0
\end{pmatrix},\ 
\gamma_s(0)=
\begin{pmatrix}
    \cos\theta_0\\ \sin\theta_0
\end{pmatrix},\ 
\gamma_s(L)=\begin{pmatrix}
    \cos\theta_1\\ \sin\theta_1
\end{pmatrix}.
\end{equation}
This is similar to the set $\A(\gamma)$ considered earlier.
Note that this formulation does not lose generality thanks to geometric invariances.
As mentioned before, there always exists a minimiser of the bending energy $\ef{B}$ in $\A^L_{\theta_0,\theta_1,\ell}$, and also every minimiser (or indeed critical point) is a smooth elastica.
However, the minimisers are not necessarily unique, and their shape is in general hard to rigorously detect just from the given boundary data.
For results with some special boundary data, see e.g.\ \cite{miura2021optimal}, which has an application to the flow; and \cite{dall2023elastic}, in which an interesting family of Noether identities have been derived, and an ordering of the graphical minimisers is obtained (note that there is no comparison principle in general for elastica).

In \cite{miu20} the first author gave a detailed analysis of global minimisers in the \emph{straightening problem}, where $\ell$ is close to $L$, for generic boundary angles $\theta_0,\theta_1\in(-\pi,\pi)\setminus\{0\}$.
However, uniqueness was obtained only in the convex case $\theta_0\theta_1<0$ for a technical reason.
Armed with Theorem \ref{thm:main_length_variable}, now we can extend it to the remaining non-convex case $\theta_0\theta_1>0$ almost directly.
As a result, we obtain the following

\begin{theorem}\label{thm:straightening}
    Let $L>0$ and $\theta_0,\theta_1\in(-\pi,\pi)\setminus\{0\}$.
    Then there exists $\ell^*\in(0,L)$ such that for any $\ell\in(\ell^*,L)$ the bending energy $\ef{B}$ has a unique minimiser $\gamma_\ell$ in the set $\A^L_{\ell,\theta_0,\theta_1}$.
    In addition:
        \begin{itemize}
            \item[(a)] If $\theta_0\theta_1<0$, then the above $\gamma_\ell$ has strictly monotone tangential angle from $\theta_0$ to $\theta_1$; in particular, $\gamma_\ell$ has no inflection point.
            \item[(b)] If $\theta_0\theta_1>0$, then the above $\gamma_\ell$ has strictly monotone curvature with a unique zero; in particular, $\gamma_\ell$ has exactly one inflection point.
        \end{itemize}
    Finally, in the limit $\ell\to L$, the rescaled curve defined by
        \begin{equation*}
            \hat{\gamma}_\ell(s):=\frac{1}{\varepsilon_\ell}\gamma_\ell(\varepsilon_\ell s), \quad \varepsilon_\ell:=\frac{L-\ell}{4\sqrt{2}(\sin^2\frac{\theta_0}{4}+\sin^2\frac{\theta_1}{4})},
        \end{equation*}
    locally smoothly converges to the borderline elastica $\gamma_B^{\theta_0}$ defined below.
    Similar convergence behaviour holds at the other endpoint by symmetry.
\end{theorem}

Here the \emph{borderline elastica with initial angle} $\theta_0\in(-\pi,\pi)\setminus\{0\}$ is explicitly defined by
    \begin{equation}
        \gamma_B^{\theta_0}(s):=\int_0^s
        \begin{pmatrix}
            \cos\varphi(u+u_0)\\
            (\sgn{\theta_0})\sin\varphi(u+u_0)
        \end{pmatrix}
        du, \quad \varphi(u):= 4\arctan\left(e^{-\frac{u}{\sqrt{2}}}\right),
    \label{EQborderintro}
    \end{equation}
where $u_0:=-\sqrt{2}\log(\tan\frac{|\theta_0|}{4})$, which is the unique positive number such that $\varphi(u_0)=|\theta_0|$.
This curve can be characterised as the unique arc-length parametrised elastica $\gamma:[0,\infty)\to\R^2$ with $\lambda=1$ and $\gamma(0)=(0,0)^\top$, $\gamma'(0)=(\cos\theta_0,\sin\theta_0)^\top$,  $\lim_{s\to\infty}\gamma'(s)=(1,0)^\top$.
See \cite{miu20} for details.

We emphasise that we will apply the length-variable result in Theorem \ref{thm:main_length_variable} to prove the fixed-length result in Theorem \ref{thm:straightening}.
This is because the singular limit approach developed in \cite{miu20} is heavily based on the analysis of the length-penalised problem, which is then applied to analyse the original fixed-length problem.

Now we discuss the idea of our proof, focusing on Theorem \ref{thm:main_fixed_length}.
The main tool is the classically known fact that planar elasticae are classified in terms of Jacobian elliptic integrals and functions (Theorem \ref{thm:planar_explicit}).
Although we already have this complete classification of critical points, it is almost infeasible to find all solutions to the general clamped boundary value problem by solving the associated system of transcendental equations (see e.g.\ \cite[Theorem 2.2]{Linner1998}) and also to detect global minimisers by computing their energy.
This point is one of the common difficulties in elastica theory, in spite of the availability of the classification.

Our main idea provides a new way to circumvent this common difficulty, partly inspired by previous work \cite{MR4877595}.
More precisely, we formulate and solve a tailor-made family of free boundary problems such that every elastica with non-constant monotone curvature appears (or part thereof) as the unique minimiser of one of those problems.
Our choice of free boundary conditions is given in terms of right-angle conditions on fixed parallel lines.
Then any critical point satisfies the additional so-called `no flux' boundary condition that the endpoints must be vertices (Lemma \ref{lem:noflux}).
This gives a very effective reduction of possible minimisers since in this case all the critical points are ``computable'', in the sense that they involve only complete elliptic integrals.
In this way we obtain unique explicit minimisers; along the way we also employ the recently-developed ``cut-and-paste'' trick \cite{MiuraYoshizawa2024crelle}.
By varying the free-boundary parameters we can exhaust all the elasticae under consideration, with the only exception being the borderline case.

Borderline elasticae require special treatment.
They may be regarded as the degenerate case where the horizontal intercept of one of the parallel lines diverges to infinity.
Then we encounter several additional issues.
First, it is unclear how to appropriately incorporate the length constraint since in this case admissible curves must have infinite length.
A simple idea would be to instead consider a length-penalised problem for $\ef{E}^\lambda$ but then the energy diverges since the length is infinite.
An additional difficulty is that the divergence of the free-boundary line loses one boundary condition.
To overcome those issues simultaneously, we introduce a suitably adapted energy $\tilde{\ef{E}}^\lambda$ (see Theorem \ref{thm:uniqminfbpborderline}).
This energy admits appropriate infinite-length curves as competitors, and also reflects the desired boundary condition (at infinity) in terms of its integrability.
In particular, half of a borderline elastica appears as a unique minimiser of the adapted energy $\tilde{\ef{E}}^\lambda$.
Not only that, the adapted energy is chosen so that any finite-length subarc of the minimiser of $\tilde{\ef{E}}^\lambda$ is also minimal for the original energy $\ef{E}^\lambda$.
These are the crucial points of the proof.

We finally remark that minimality and stability of spatial elasticae are significantly different from the planar case and less-understood, see e.g.\ \cite[Section 7]{MiuraYoshizawa2024crelle} and references therein.

\begin{remark}
    Another possible approach to uniqueness results as in Theorem \ref{thm:straightening} is to investigate the attainable boundary conditions for pieces of elasticae that are already known to be unique minimisers. 
    Our Theorem \ref{thm:main_fixed_length} would be useful in this approach, and we expect that, at least in case (b) of Theorem \ref{thm:straightening}, uniqueness could be directly obtained by applying Theorem \ref{thm:main_fixed_length} with an explicit estimate for $\ell^*$. 
    However, this would require detailed computations involving (incomplete) elliptic integrals and functions, so we do not enter into those details in this paper.
\end{remark}

\subsection*{Acknowledgements}
The first author is supported by JSPS KAKENHI Grant Numbers 20K14341, 21H00990, and 23H00085.
The second author was partially supported by AEGiS Advance grant number PR6064 from the University of Wollongong.
Much of this work was completed while the second author was visiting the first at the Tokyo Institute of Technology.
He is grateful for their support and kind hospitality.

\section{Preliminaries}

Let $\ell\in\R$.
We say that a curve $\gamma$ has \emph{free boundary (with horizontal extent $\ell$)} if $\gamma\in\SF_\ell^+$ or $\gamma\in\SF_\ell^-$ where
\begin{align*}
    \SF_\ell^\pm
    := \Bigg\{\gamma \in \bigcup_{L>0} \SolnSpace\,:\, 
    \begin{split}
        (\gamma(0), e_1) &= 0,\ (\gamma(L),e_1) = \ell
    \\
    \gamma'(0)&= e_1,\  \gamma'(L) = \pm e_1 
    \end{split}\ 
    \Bigg\}.
\end{align*}
Here $e_1 = (1,0)^\top$ and $e_2 = (0,1)^\top$ denote the canonical basis of $\R^2$.
The space $\SF_\ell^\pm$ does not restrict the length of its members.
It requires $\gamma(0)$ to lie on the vertical line $t\mapsto te_2$ and $\gamma(L)$ to lie on the vertical line $t\mapsto \ell e_1 + te_2$; and it requires $\gamma$ to satisfy a perpendicularity condition on these lines.
The lines support the boundary of $\gamma\in\SF_\ell^\pm$, so we call them \emph{support lines}.
Note that $\ell$ may be negative; this simply means that the line $t\mapsto \ell e_1 + te_2$ is to the left of the starting line $t\mapsto te_2$.

Now we observe that the well-known first variation formulae yield an emergent boundary condition for critical points with free boundary.
For the moment, let us consider a general (smooth) planar curve $\gamma:\bar{I}\to\R^2$, where $I = (a,b)$, which is immersed ($|\partial_x\gamma(x)| > 0$) but not necessarily in the arc-length parametrisation.
The length of $\gamma$ in its given parametrisation is
$$
    \ef{L}[\gamma] := \int_I |\partial_x\gamma(x)|\,dx.
$$
The arc-length derivative $\partial_s$ and measure $ds$ are defined by $\partial_s = |\partial_x\gamma|^{-1}\partial_x$ and $ds:=|\partial_s\gamma|dx$, respectively.
The unit tangent vector is given by $T := \partial_s\gamma$, and the unit normal by its counterclockwise $\frac{\pi}{2}$-rotation $N:=\operatorname{rot}T$.
The curvature scalar $k$ is then defined so that $\partial_s^2\gamma=kN$.
Accordingly the bending energy is given by 
$$\ef{B}[\gamma] = \int_I |\partial_s^2\gamma|^2\,ds.$$
Recall also the Frenet--Serret formula $\partial_sN=-kT$.
Hereafter we essentially only work in the arc-length parametrisation, so we  use $'$ to denote the (arc-length) derivative.

For an arc-length parametrised curve $\gamma\in W_\mathrm{arc}^{2,2}(0,L;\R^2)$, the following formula holds in any direction $\eta\in W^{2,2}(0,L;\R^2)$ (see e.g.\ \cite[Lemma A.1]{MiuraYoshizawa2024AMPA}):
\begin{equation}\label{eq:first_variation}
    \frac{d}{d\varepsilon}\bigg|_{\varepsilon=0}\ef{E}^\lambda[\gamma+\varepsilon\eta]=\int_0^L \Big( -3 k^2 (T, \eta') +2k(N, \eta'') + \lambda (T,\eta')  \Big)\,ds.
\end{equation}
If $\gamma$ is sufficiently smooth, then we can further perform integration by parts to deduce that 
\begin{align}\label{eq:first_variation_2}
    \begin{split}
        \frac{d}{d\varepsilon}\bigg|_{\varepsilon=0}\ef{E}^\lambda[\gamma+\varepsilon\eta]= &\int_0^L (2k'' + k^3-\lambda k)(N,\eta)\, ds \\
    &+ \big[ (2kN,\eta')-(2k'N+(k^2+\lambda)T,\eta) \big]_0^L.
    \end{split}
\end{align}

\begin{lemma}[No flux condition]
\label{lem:noflux}
Let $\lambda\in\R$ and $\gamma\in \SF_\ell^\pm$ be a critical point of $\ef{E}^\lambda$ in $\SF_\ell^\pm$ in the sense that the first variation \eqref{eq:first_variation} vanishes for any admissible direction $\eta\in W^{2,2}(0,L;\R^2)$ such that $(\eta(0),e_1)=(\eta(L),e_1)=(\eta'(0),e_2)=(\eta'(L),e_2)=0$.
Then $\gamma$ is a $\lambda$-elastica (and hence smooth).
In addition,
$$k'(0) = k'(L) = 0.$$
\end{lemma}

\begin{proof}
The fact that $\gamma$ is a smooth elastica follows by a known bootstrap argument, where we only need to require \eqref{eq:first_variation} to vanish for all $\eta\in C^\infty_c(0,L;\R^2)$ (see e.g.\ \cite{MiuraYoshizawa2024AMPA}).
This means that the integral term always vanishes in \eqref{eq:first_variation_2}.
Hence by assumption, for any admissible $\eta$,
\[
[ (2kN,\eta')-(2k'N+(k^2+\lambda)T,\eta) ]_0^L=0.
\]
In particular, if we take $\eta = \phi N$, where $\phi$ is a smooth cutoff function satisfying $\phi(0) = 1$, $\phi'(0) = 0$, and $\phi(s) = 0$ for $s>\frac{L}{2}$,
then
\[
0=[ 2k\phi'-2k'\phi ]_0^L = 2k'(0),
\]
as required.
By symmetry we also have $k'(L) = 0$.
\end{proof}

One difficult aspect of Euler's elastica problem is that the bending energy has a significant degeneracy.
That is particularly relevant to us here, as our main goal is to establish uniqueness of elastica in various senses.
To be precise, any isometry $M:\R^2\rightarrow\R^2$ preserves the bending energy and the length, as does any reparametrisation $\phi:\overline{J}\mapsto\overline{I}$; we have $\ef{B}[M\circ\gamma\circ\phi] = \ef{B}[\gamma]$ and $\ef{L}[M\circ\gamma\circ\phi] = \ef{L}[\gamma]$.

We work in the arc-length parametrisation throughout, and we aim to prove uniqueness in this parametrisation.
This almost removes the degeneracy in parametrisation (only shifts and reflections remain), and so it does not represent a significant obstacle in our analysis.
However for isometries the situation is more complicated.
While both bending and length are invariant under the action of an isometry $M$, the image $\gamma(I)$ is not invariant; that is, $(M\circ\gamma)(I) \ne \gamma(I)$.
This means that many isometries will not be \emph{admissible} for our free boundary problems.
To be more precise, we have the following result.

\begin{lemma}
\label{lem:invariancesM}
The admissible subgroup $\SM$ of $\SF_\ell^\pm$-invariant isometries consists of the vertical reflection $(x_1,x_2)^\top\mapsto (x_1,-x_2)^\top$, the vertical translation $(x_1,x_2)^\top\mapsto (x_1,x_2+\alpha)^\top$ with arbitrary $\alpha\in\R$, and their compositions.
\end{lemma}

The remaining action we wish to discuss is that of scaling.
Scaling is not an isometry, but we shall show that it nevertheless maps from an elastica to an elastica, and so it is important for us to consider.
Given a curve $\gamma$, we define its rescaling by a factor $\Lambda>0$ to be
$$\gamma_\Lambda(s) := \Lambda \gamma(s/\Lambda).$$
The division by $\Lambda$ is to ensure that $\gamma_\Lambda$ remains parametrised by arc-length.
Then straightforward calculation yields $ds_\Lambda = \Lambda ds$,  $\kappa_\Lambda = \Lambda^{-1}\kappa$, and hence
\begin{equation}\label{eq:scalingBL}
    \ef{B}[\gamma_\Lambda] = \Lambda^{-1} \ef{B}[\gamma]\ \text{and}\ 
\ef{L}[\gamma_\Lambda] = \Lambda \ef{L}[\gamma].
\end{equation}
For the energy $\ef{E}^\lambda$, this means that scaling introduces a competition between the bending and the length.
This implies the following.

\begin{lemma}
\label{lem:scaling}
Suppose $\gamma:[0,L]\rightarrow\R^2$ is a $\lambda$-elastica.
Let $\Lambda>0$.
Then the rescaled curve $\gamma_\Lambda$ is a $(\lambda\Lambda^{-2})$-elastica.
\end{lemma}

\begin{proof}
Observe that $2k_\Lambda'' + k_\Lambda^3
= \Lambda^{-3} (2k'' + k^3)
= \Lambda^{-3} \lambda k
= (\Lambda^{-2} \lambda) k_\Lambda$.
\end{proof}

One of the most remarkable achievements of classical investigation into planar elasticae was their complete classification, through the use (and in the first instance, invention) of elliptic functions.
Very roughly speaking, this follows by integrating the Euler--Lagrange equation once after multiplying by $k_s$, using the Picard--Lindel\"of theorem to obtain existence, and then explicit expressions through properties of the elliptic functions and consideration of many cases.
For more detailed arguments, see e.g.\ \cite[Appendix B]{MuellerRupp2023} or \cite{MiuraYoshizawa2024AMPA}.

To state the classification we briefly recall Jacobi elliptic integrals and functions.
For more details see classical textbooks, e.g.\ \cite[Chapter XXII]{Whittaker1962} (and also \cite{Abramowitz1992}), or see \cite[Appendix A]{Miura2024survey} which collects all the necessary properties we will use here with the same conventions and notation.
The {\em incomplete elliptic integral of the first kind} $F(x,m)$ and {\em of the second kind} $E(x,m)$ with {\em parameter} $m\in(0,1)$ (squared elliptic modulus) are defined by
\begin{equation*}
  F(x,m):=\int_0^x\frac{d\theta}{\sqrt{1-m\sin^2\theta}}, \quad E(x,m):=\int_0^x\sqrt{1-m\sin^2\theta}d\theta,
\end{equation*}
respectively.
The {\em complete elliptic integral of first kind} $K(m)$ and {\em of second kind} $E(m)$ are then defined by
$$K(m):=F(\pi/2,m), \quad E(m):=E(\pi/2,m),$$
respectively.
The {\em (Jacobi) amplitude function} is defined by
$$\am(\cdot,m):=F^{-1}(\cdot,m) \quad \text{on}\ \mathbf{R}.$$
The {\em (Jacobi) elliptic functions} are then given by
\begin{align*}
  \cn(x,m)&:=\cos(\am(x,m)),\quad \sn(x,m):=\sin(\am(x,m)),\\
  \dn(x,m)&:=\sqrt{1-m\sin^2(\am(x,m))}.
\end{align*}

\begin{theorem}[Explicit parametrisations of planar elasticae]\label{thm:planar_explicit}
  Let $\gamma:[0,L]\rightarrow\R^2$ be an elastica in $\R^2$.
  Then, up to the action of the similarity group on $\R^2$ and reparametrisation,
\[
\gamma(s)=\gamma_*(s+s_0)
\]
with $s_0\in\R$, where $\gamma_*$ is one of the following five parametrisations:
  \begin{enumerate}
    \item[\upshape(I)] \textup{(Linear elastica)}
    \begin{equation*}
      \gamma_\ell(s) =
      \begin{pmatrix}
      s  \\
      0
      \end{pmatrix}.
    \end{equation*}
    In this case, $\theta_\ell(s)=0$ and $k_\ell(s)=0$.
    \item[\upshape(II)] \textup{(Wavelike elastica)}
    There exists $m \in (0,1)$ such that
    \begin{equation*}
      \gamma_w(s)=\gamma_w(s,m) =
      \begin{pmatrix}
      2 E(\am(s,m),m )-  s  \\
      -2 \sqrt{m} \cn(s,m)
      \end{pmatrix}.
    \end{equation*}
    In this case, $\theta_w(s)=2\arcsin(\sqrt{m}\sn(s,m))$ and $k_w(s) = 2 \sqrt{m} \cn(s,m).$
    \item[\upshape(III)] \textup{(Borderline elastica)}
    \begin{equation*}
      \gamma_b(s) =
      \begin{pmatrix}
      2 \tanh{s} - s  \\
      - 2 \sech{s}
      \end{pmatrix}.
    \end{equation*}
    In this case, $\theta_b(s)=2\arcsin(\tanh{s})$ and $k_b(s) = 2\sech{s}.$
    \item[\upshape(IV)] \textup{(Orbitlike elastica)}
    There exists $m \in (0,1)$ such that
    \begin{equation*}
      \gamma_o(s)=\gamma_o(s,m) = \frac{1}{m}
      \begin{pmatrix}
      2 E(\am(s,m),m)  + (m-2)s \\
      - 2\dn(s,m)
      \end{pmatrix}.
    \end{equation*}
    In this case, $\theta_o(s)=2\am(s,m)$ and $k_o(s) = 2 \dn(s,m).$
    \item[\upshape(V)] \textup{(Circular elastica)}
    \begin{equation*}
      \gamma_c(s) =
      \begin{pmatrix}
      \cos{s}  \\
      \sin{s}
      \end{pmatrix}.
    \end{equation*}
    In this case, $\theta_c(s)=s+\pi/2$ and $k_c(s)=1$.
  \end{enumerate}
  Here the tangential angle $\theta_*$ and the curvature $k_*$ satisfy the relations $\partial_s\gamma_*=(\cos\theta_*,\sin\theta_*)^\top$ and $\partial_s\theta_*=k_*$.
\end{theorem}

\begin{remark}
    The tangential angle of the borderline elastica is twice the Gudermannian function, and thus has several equivalent expressions: $\theta_b(s)=2\am(s,1)=2\arcsin(\tanh{s})=2\arctan(\sinh{s})=4\arctan(e^s)-\pi$.
    In particular, up to isometry and reparametrisation, the curve $\gamma_B^{\theta_0}$ defined in the introduction \eqref{EQborderintro} is given by $\sqrt{2}\gamma_b(s/\sqrt{2})$.
\end{remark}

\begin{definition}\label{def:supercritical}
    A wavelike elastica is \emph{supercritical} (resp.\ \emph{critical}, \emph{subcritical}) if the corresponding parameter $m\in(0,1)$ satisfies $m>m_0$ (resp.\ $m=m_0$, $m<m_0$), where $m_0\in(0,1)$ is the unique root of the decreasing function $m\mapsto 2E(m)-K(m)$.
\end{definition}

\begin{remark}
    The wavelike elastica with $m_0\simeq 0.8261$ produces the so-called \emph{figure-eight elastica}, which is the only non-circular closed planar elastica.
\end{remark}

\begin{remark}\label{rem:multiplier}
    We can directly compute the multipliers of the base curves:
    \begin{itemize}
        \item[(I)] $\gamma_\ell$ is a $\lambda$-elastica for any $\lambda\in\R$.
        \item[(II)] $\gamma_w=\gamma_w(\cdot,m)$ is a $\lambda$-elastica for $\lambda=2(2m-1)$.
        \item[(III)] $\gamma_b$ is a $\lambda$-elastica for $\lambda=2$.
        \item[(IV)] $\gamma_o=\gamma_o(\cdot,m)$ is a $\lambda$-elastica for $\lambda=2(2-m)$.
        \item[(V)] $\gamma_c$ is a $\lambda$-elastica for $\lambda=1$.
    \end{itemize}
    Moreover, the sign of the multiplier is preserved under similarity actions.
    Therefore, the multiplier can be negative if and only if $\gamma$ is either a linear elastica or a wavelike elastica with parameter $m\in(0,\frac{1}{2})$.
    In addition, any $0$-elastica (also called \emph{free elastica}) is linear or wavelike with $m=\frac{1}{2}$.
    Since $m_0>\frac{1}{2}$ (see \cite{Miura23} for a proof), the class of subcritical wavelike elasticae contains all elasticae with non-positive multipliers $\lambda\leq0$.
    Therefore, the assumption in Theorem \ref{thm:main_length_variable} automatically implies that $\lambda>0$.
    On the other hand, for any $\lambda$-elastica $\gamma$ with negative multiplier $\lambda<0$, we can always construct a large-length competitor in $\widetilde{\mathcal{A}}(\gamma)$ such that $E_\lambda\to-\infty$, and hence $\gamma$ does not have the minimality as in Theorem \ref{thm:main_length_variable}.
\end{remark}

\section{Free boundary problems with fixed length}
\label{sec:freebdy}

In this section we prove Theorem \ref{thm:main_fixed_length}.
Non-constant monotone curvature elasticae come in three flavours: an arc of a wavelike elastica (type (II)), of a borderline elastica (type (III)), or of an orbitlike elastica (type (IV)).
For all wavelike and orbitlike elasticae, we present distinct free boundary minimisation problems that each type uniquely solves.
The borderline elastica may be regarded as a degenerate case with infinite horizontal extent.

Let us begin with the wavelike elastica.

\begin{theorem}
\label{thm:uniqminfixedL}
    Let $\ell,L\in\R$ be fixed parameters satisfying $-L<\ell<L$.
    Consider the set $\SA^L$ of admissible curves given by
    \begin{align*}
    \SA^L = \Big\{ \gamma\in \SF_\ell^+ \,:\, \ef{L}[\gamma]=L \Big\}.
    \end{align*}
    Then there exists a smooth minimiser $\gamma_{\ell,L}$ of the bending energy $\ef{B}:\SA^L\rightarrow(0,\infty)$.
    
    In addition, there is a unique parameter $m\in(0,1)$, depending only on the ratio $\ell/L$, such that any minimiser is parametrised by
    \begin{equation}\label{eq:wavelike_fixedL_parametrisation}
        \gamma_{\ell,L}(s)=\tfrac{L}{2K(m)} M\gamma_w(\tfrac{2K(m)}{L}s,m)
    \end{equation} 
    for some $M\in\SM$ (see Lemma \ref{lem:invariancesM}), where $\gamma_w$ is the wavelike elastica in Theorem \ref{thm:planar_explicit}.
    In particular, the minimiser is unique up to vertical translations and reflection.

    Moreover, the inverse of 
    \[m \mapsto \frac{2E(m)}{K(m)}-1\]
    is a strictly increasing bijective function $\phi:(-1,1)\to(0,1)$ such that $\phi(\ell/L) = m$ with $\phi(0) = m_0$.
    Further, if $\ell>0$, then $m\in(m_0,1)$; if $\ell<0$ then $m\in(0,m_0)$; and if $\ell=0$ then $m=m_0$ and hence $\gamma_{0,L}$ is (half of) the figure-eight elastica.
\end{theorem}

\begin{proof}
The direct method in the calculus of variations implies the existence of a minimiser; for example follow the proof of \cite[Proposition 4.1]{MR4852697} with the additional argument that, without loss of generality, we can take a minimising sequence $\{\gamma_j\}$ satisfying $\gamma_j(0)=(0,0)^\top$ up to vertical translations.
Then in view of the standard multiplier method, the minimiser must be a critical point $\ef{E}^\lambda$ for some $
\lambda\in\R$ in the sense of Lemma \ref{lem:noflux}.
(See \cite[Appendix A]{MR4852697} for basic facts, and also \cite[Appendix A]{MR4877595} for a similar boundary condition.)
Hence the minimiser is a smooth elastica.

In what follows we prove uniqueness.
Any minimiser $\gamma_{\ell,L}$ must belong to the classification given by Theorem \ref{thm:planar_explicit}.
This means that it has parametrisation 
\begin{equation}
\label{eqn:generalformUniqwavelikeproof}
\gamma_{\ell,L}(s) = M\Lambda\gamma_*\big((s+s_0)\Lambda^{-1}\big),
\end{equation}
where $\Lambda>0$, $M:\R^2\to\R^2$ is an isometry, and $\gamma_*$ is one of $\gamma_\ell$, $\gamma_w$, $\gamma_b$, $\gamma_o$, or $\gamma_c$ corresponding to cases (I)--(V) in Theorem \ref{thm:planar_explicit}.
Note that the action of reparametrisation is already taken into consideration, since the speed of the parameter is fixed by arc-length parametrisation, and the orientation can be incorporated into the choices of $M$ and $s_0$ thanks to periodicity and symmetry.
This formula therefore takes into account all of the degeneracy mentioned earlier.

The goal for the remainder of the proof is to systematically reduce and simplify the freedom and form of \eqref{eqn:generalformUniqwavelikeproof}.

{\it Ruling out Case (I)}.
Linear elasticae (I) satisfying the boundary condition have $L = \ell$ and so they are not members of $\SA^L$ since here $\ell<L$.

{\it Ruling out Case (III)}.
The borderline case $\gamma_*=\gamma_b$ is also impossible since $\gamma_{\ell,L}$ needs to satisfy the no flux condition in Lemma \ref{lem:noflux} at both the endpoints, but the curvature $k_b$ of $\gamma_b$ has only one vertex.

{\it Ruling out Case (IV).}
We prove by contradiction.
Suppose that $\gamma_*=\gamma_o$.
By Lemma \ref{lem:noflux} a candidate minimiser $\gamma_{\ell,L}$ must have vertices on its boundary.
By periodicity of $k_o$, we may assume that $s_0\in\{0,\Lambda K(m)\}$, and $L=j\Lambda K(m)$ for some positive integer $j$.
In addition, since $\theta_o(iK(m))=i\pi$ holds for any integer $i$, the boundary condition $\gamma_{\ell,L}'(0)=\gamma_{\ell,L}'(L)=e_1$ implies that not only $j$ is even but also $\gamma_{\ell,L}'(\frac{L}{2})=-e_1$.
This means that, if we reflect vertically the arc $\gamma_{\ell,L}|_{[\frac{L}{2},L]}$ while keeping the former half $\gamma_{\ell,L}|_{[0,\frac{L}{2}]}$ as it is, we can construct a new (admissible) minimiser.
However this flipped curve has a curvature jump at $s=\frac{L}{2}$.
This contradicts the smoothness of minimisers.
(This cut-and-paste argument is a variant of the method developed in \cite{MiuraYoshizawa2024crelle}.)

{\it Ruling out Case (V).}
For circular elasticae, the boundary condition immediately implies that $\ell=0$ and $\gamma_{\ell,L}$ to be given by a whole circle, possibly multiply covered.
Then the same reflection argument as in case (IV) again implies a contradiction.

{\it Case (II).} The only remaining possibility is the wavelike case $\gamma_*=\gamma_w$.
By Lemma \ref{lem:noflux} and periodicity, we deduce that $s_0\in\{0,2K(m)\Lambda\}$ and $L=2j\Lambda K(m)$ for some positive integer $j$.
Up to a vertical reflection (which does not lose the admissibility) we may assume that $s_0=0$.
Recall that in general the isometry $M$ is of the form $M=A+b$
for a linear isometry $A$ and a translation vector $b$.
By the boundary condition $\gamma_{\ell,L}'(0)=e_1$, we have $A\gamma_w'(0)=e_1$, i.e., $A$ is either the identity or the vertical reflection.
(At this point $\gamma_{\ell,L}'(L)=e_1$ holds automatically.)
Then, since $(\gamma_w(0),e_1)=0$, the boundary condition $(\gamma_{\ell,L}(0),e_1)=0$ implies that $(b,e_1)=0$, i.e., $b$ is a vertical translation.
In summary, we deduce that
$$\gamma_{\ell,L}(s)=\Lambda M\gamma_w(s/\Lambda,m)$$ 
for some $M\in\SM$ and some $m\in(0,1)$, with the scaling factor $\Lambda = \frac{L}{2jK(m)}$.

Moreover, if $j\geq2$, then by periodicity the dilated subarc $\tilde{\gamma}(s):=j\gamma_{\ell,L}|_{[0,\frac{L}{j}]}(s/j)$ is still admissible and has less energy than $\gamma_{\ell,L}$ by the scaling property of the bending energy $\ef{B}$, but this contradicts the minimality of $\gamma_{\ell,L}$.
Hence $j=1$, i.e.,
$$\Lambda=L/(2K(m)).$$

Finally, we detect a unique parameter $m\in(0,1)$.
To this end we use the remaining boundary condition $(\gamma_{\ell,L}(L),e_1)=\ell$.
This condition, combined with all the obtained properties (as well as some basic properties like $\am(iK(m),m) = i\pi/2$ and $E(i\pi/2,m) = iE(m)$ for $i\in \Z$), implies that
$2\Lambda(2E(m)-K(m))=\ell.$
Thus
\begin{equation}
\label{eqn:monotoneqLfixed}
    \frac{2E(m)}{K(m)} - 1 = \frac{\ell}{L}.
\end{equation}
It is well known (and not difficult to show, see e.g.\ \cite{Yoshizawa2022,MuellerRupp2023,Miura23}) that the map $f:m\mapsto 2E(m)/K(m) - 1$ is strictly monotone and bijective from $(0,1)$ to $(-1,1)$, and has the unique root $m=m_0$.
This implies the uniqueness of $m\in(0,1)$, and also that the inverse $\phi=f^{-1}$ satisfies all the desired properties.
The proof is complete.
\end{proof}

In particular, this implies unique minimality in the wavelike case.

\begin{corollary}\label{cor:unique_wavelike}
    Let $m\in(0,1)$ and $\gamma$ be an elastica given by a subarc of a similarity image of the half-period wavelike elastica $\gamma_w(\cdot,m)|_{[0,2K(m)]}$.
    Then $\gamma$ is the unique minimiser of $\ef{B}$ in $\mathcal{A}(\gamma)$.
\end{corollary}

\begin{proof}
    In Theorem \ref{thm:uniqminfixedL}, if we fix the left endpoint of minimisers, then the vertical translation is uniquely determined.
    Moreover, since
    $$(\gamma_w(0,m),e_2)=-2\sqrt{m}\neq 2\sqrt{m}=(\gamma_w(2K(m),m),e_2),$$
    if we additionally fix the right endpoint, then the vertical reflection is also uniquely determined.
    Hence the minimiser must be unique.
    Therefore, considering $L:=2K(m)$ and $\ell:=2K(m)\phi^{-1}(m)$ in Theorem \ref{thm:uniqminfixedL}, we deduce that the curve $\tilde{\gamma}:=\gamma_w(\cdot,m)|_{[0,2K(m)]}$ is the unique minimiser in $\SA(\tilde{\gamma})$.
    This also implies unique minimality of any subarc of $\tilde{\gamma}$, since otherwise we could replace the subarc with a new curve to construct an admissible curve $\xi\in\SA(\tilde{\gamma})$ such that $\xi\neq\tilde{\gamma}$ and $\ef{B}[\xi]\leq\ef{B}[\tilde{\gamma}]$, which contradicts unique minimality of $\tilde{\gamma}$.
    As the property of unique minimality is invariant with respect to similarity actions, the same is true for any similarity image of $\tilde{\gamma}$.
\end{proof}

\begin{remark}
    Note that, as similarity sends unique minimisers to unique minimisers, we may without loss of generality fix e.g. $L=1$ for these uniqueness results.
    Since this does not simplify the arguments significantly however, we have chosen not to pursue this normalisation.
\end{remark}

Next we discuss the orbitlike case.

\begin{theorem}
\label{thm:uniqminfixedLorbitcase}
    Let $\ell,L\in\R$ be fixed parameters satisfying $-L<\ell<0$.
    Consider the set $\SB^L$ of admissible curves given by
    \begin{align*}
    \SB^L = \Big\{ \gamma\in \SF_\ell^- \,:\, \ef{L}[\gamma]=L \Big\}.
    \end{align*}
    Then there exists a smooth minimiser $\gamma_{\ell,L}$ of the bending energy $\ef{B}:\SB^L\rightarrow(0,\infty)$.
    
    In addition, there is a unique parameter $m\in(0,1)$, depending only on the ratio $\ell/L$, such that any minimiser is parametrised by
    \begin{equation}\label{eq:orbitlike_fixedL_parametrisation}
        \gamma_{\ell,L}(s)=\tfrac{L}{K(m)} M\gamma_o(\tfrac{K(m)}{L}s,m)
    \end{equation}
    for some $M\in\SM$ (see Lemma \ref{lem:invariancesM}), where 
    $\gamma_o$ is the orbitlike elastica in Theorem \ref{thm:planar_explicit}.
    In particular, the minimiser is unique up to vertical translations and reflection.
    
    Moreover, the inverse of
    \[m \mapsto \frac{2E(m)}{mK(m)}+1-\frac{2}{m}\] 
    is a strictly decreasing bijective function $\psi:(-1,0)\to(0,1)$ such that $\psi(\ell/L) = m$.
\end{theorem}

\begin{proof}
    As in the proof of Theorem \ref{thm:uniqminfixedL} the existence of a minimiser follows by
    using the direct method, and any minimiser must be given by
    \begin{equation}
    \label{eqn:generalformUniqorbitlikeproof}
    \gamma_{\ell,L}(s) = M\Lambda\gamma_*\big((s+s_0)\Lambda^{-1}\big),
    \end{equation}
    where $\Lambda>0$, $M:\R^2\to\R^2$ is an isometry, and $\gamma_*$ is one of $\gamma_\ell$, $\gamma_w$, $\gamma_b$, $\gamma_o$, or $\gamma_c$ in Theorem \ref{thm:planar_explicit}.
    We now work to reduce the space of possible minimisers.
    
    {\it Ruling out Cases (I), (III), (V)}.
    Linear elasticae do not change angle, and so
    can not be members of $\SB^L$.
    Circular elasticae do not satisfy
    the boundary condition unless $\ell=0$, so are not members of $\SB^L$ since here $\ell<0$.
    Borderline elasticae are also ruled out for the same reason as in the proof of Theorem \ref{thm:uniqminfixedL}.
    
    {\it Ruling out Case (II).}
    For wavelike elasticae, arguing as in the proof of Theorem
    \ref{thm:uniqminfixedL}, Lemma \ref{lem:noflux} implies that the vertices of
    $M\Lambda\gamma_w$ must appear at the endpoints.
    However, at all vertices the tangent vector of $M\Lambda\gamma_w$ is in a same direction, so for the boundary conditions of $\SB^L$, the wavelike case can not occur.
    
    {\it Case (IV).}
    As in the proof of Theorem \ref{thm:uniqminfixedL} for Case (II), Lemma \ref{lem:noflux} with periodicity and symmetry implies that
    $$\gamma_{\ell,L}(s) = M\Lambda \gamma_o((s+s_0)\Lambda^{-1})$$
    for some $s_0\in\{0,K(m)\Lambda\}$, some isometry $M:\R^2\to\R^2$
    such that
    \begin{itemize}
        \item if $s_0=0$, then $M\in\SM$;
        \item if $s_0=K(m)\Lambda$, then $M=M'N$ where $M'\in\SM$ and $N:(x_1,x_2)^\top\mapsto(-x_1,-x_2)^\top+\Lambda\gamma_o(K(m))$,
    \end{itemize}
    and $L=jK(m)\Lambda$ for some integer $j\geq1$.
    The scaling argument as before implies $j=1$, that is, 
    $$\Lambda=L/K(m).$$

    Now, assuming that $s_0=0$, we prove the uniqueness of $m\in(0,1)$.
    By the boundary condition $(\gamma_{\ell,L}(L),e_1)=\ell$, we compute
    $\frac{\Lambda}{m}(2E(m) + (m-2)K(m))=\ell$, and hence
    \begin{equation}
    \label{eqn:monotoneqLfixedorbitlike}
        \frac{2E(m)}{mK(m)} + 1 - \frac2m = \frac{\ell}{L}.
    \end{equation}
    We now show that the map 
    $$f:m\mapsto\frac{2E(m)}{mK(m)}+1-\frac{2}{m}$$ 
    is strictly decreasing.
    Since (let us omit the argument $m$ for brevity)
    $$K'=\frac{E-(1-m)K}{2m(1-m)}, \quad E'=\frac{E-K}{2m},$$ 
    we have $$f'=\frac{2}{m^2K^2}(mE'K-EK-mEK'+K^2)=\frac{2}{m^2K^2}\frac{(1-m)K^2-E^2}{2(1-m)}.$$ 
    It now suffices to check the negativity of $g(m):=(1-m)K(m)^2-E(m)^2$.
    Note that $g(0)=0$ as $K(0)=E(0)=\frac{\pi}{2}$.
    In addition,
    \begin{align*}
        g'&=-K^2+(1-m)2KK'-2EE'\\
        &=\frac{1}{m}(-mK^2+K(E-(1-m)K)-E(E-K))
        =-\frac{1}{m}(K-E)^2,
    \end{align*}
    which is negative for $m\in(0,1)$, and hence $g$ is negative for $m\in(0,1)$.
    Therefore $f$ is strictly decreasing on $(0,1)$, and in particular there is a unique $m\in(0,1)$ satisfying \eqref{eqn:monotoneqLfixedorbitlike}.

    Now let us show that $s_0=K(m)\Lambda$ can not occur.
    If this were the case, then
    \begin{align*}
        (\gamma_{\ell,L}(L),e_1)-(\gamma_{\ell,L}(0),e_1) &=\big( N\Lambda\gamma_o(2K(m)),e_1 \big)-\big( N\Lambda\gamma_o(K(m)),e_1 \big) \\
        &=  \Lambda\big( -\gamma_o(2K(m))+\gamma_o(K(m)),e_1 \big) - 0 \\
        &= - \frac{\Lambda}{m}( 2E(m) + (m-2)K(m)),
    \end{align*}
    which is strictly positive by the previous paragraph.
    This contradicts the boundary condition of $\SB^L$ with $\ell<0$.
    Therefore $s_0=0$ is the only possible case.

    Finally, by noting that $f=1+4\frac{E'}{K}$ and recalling that $E'(0)=-\frac{\pi}{8}$ and $K(0)=\frac{\pi}{2}$, we compute $f(0+)=0$, and also by recalling that $E(1)=1$ and $K(1-)=\infty$, we have $f(1-)=-1$ (see \cite[Corollary A.8]{Miura2024survey} for details).
    Hence $f$ is bijective from $(0,1)$ to $(-1,0)$.
    The inverse $\psi:=f^{-1}$ gives the desired map.
\end{proof}

Theorem \ref{thm:uniqminfixedLorbitcase} allows us to establish unique minimality in the orbitlike case.

\begin{corollary}\label{cor:unique_orbitlike}
    Let $m\in(0,1)$ and $\gamma$ be an elastica given by a subarc of a similarity image of the half-period orbitlike elastica $\gamma_o(\cdot,m)|_{[0,K(m)]}$.
    Then $\gamma$ is the unique minimiser of $\ef{B}$ in $\mathcal{A}(\gamma)$.
\end{corollary}

The proof of Corollary \ref{cor:unique_orbitlike} follows along the same lines as the argument for Corollary \ref{cor:unique_wavelike}.

\begin{remark}\label{rem:orbit_ell_0_positive}
    The case $0<\ell<L$ can also be covered by Theorem \ref{thm:uniqminfixedLorbitcase} up to symmetry.
    The case $\ell=0$ is slightly different but it is easy to deduce by the Cauchy--Schwarz inequality that the only minimiser is a half-circle up to invariances.
\end{remark}

Finally we address the borderline case.

\begin{theorem}
\label{thm:uniqminfbpborderline}
    Let
    \begin{align*}
    \SC^N = \bigg\{\gamma \in W^{2,2}_{\text{arc}}(0,N;\R^2)\,:\, (\gamma(0), e_1) = 0,\ \gamma'(0) = e_1\bigg\},
    \end{align*}
    and consider $\SC$ the set of admissible curves given by
    \begin{align*}
    \SC = \bigg\{\gamma:[0,\infty)\to\R^2 \,:\, \forall N>0,\  \gamma|_{[0,N]}\in \SC^N \bigg\}.
    \end{align*}
    Let $\lambda>0$.
    Consider the adapted energy
    \[
    \tilde{\ef{E}}^\lambda[\gamma] := \ef{B}[\gamma] + \lambda\int_0^\infty \big( 1 + (\gamma_s, e_1) \big)\,ds.
    \]
    Then there exists a smooth minimiser $\tilde{\gamma}$ of the adapted energy $\tilde{\ef{E}}^\lambda:\SC\to(0,\infty)$.

    In addition, any minimiser is parametrised by
    \[
    \tilde{\gamma}(s) = M\gamma_b(s)
    \]
    for some $M\in\SM$ (see Lemma \ref{lem:invariancesM}), where $\gamma_b$ is the borderline elastica in Theorem \ref{thm:planar_explicit}.
    In particular, the minimiser is unique up to vertical translations and reflection.
\end{theorem}

\begin{proof}
    To begin, observe that a half-period $\tilde\gamma_b = \gamma_b|_{[0,\infty)}$ of the borderline elastica is an element of $\SC$ and has finite energy $\tilde{\ef{E}}^\lambda[\tilde\gamma_b] < \infty$.
    This can be explicitly calculated using $\theta_b(s) = 2\arcsin(\tanh s)$ and $k_b(s) = 2\sech(s)$.
    Note for later that $k_b'(0) = 0$ and this is the only vertex of $\gamma_b$.

    Now we proceed with a direct method.
    Take a minimising sequence $\{\gamma_j-\gamma_j(0)\}$ for $\tilde{\ef{E}}^\lambda$ on the (non-empty) admissible set $\SC$.
    We prove that there exists a subsequence that is $W^{2,2}$-weakly convergent to a limit $\tilde \gamma\in\SC$ on every bounded interval $K\subset\subset[0,\infty)$.
    In order to see this, we use a diagonal-type argument.
    Let $\tilde{\gamma}_j:=\gamma_j-\gamma_j(0)$, and 
    first consider the restricted sequence $\{\tilde{\gamma}_j|_{[0,1]}\}$.
    This is uniformly bounded in $W^{2,2}(0,1;\R^2)$ by the energy bound $\|\tilde{\gamma}_j''|_{[0,1]}\|_{L^2}^2=\ef{B}[\tilde{\gamma}_j]\leq\tilde{\ef{E}}^\lambda[\tilde{\gamma}_j]$ and by arc-length parametrisation.
    So it has a weakly convergent subsequence (that we use the same index for, abusing notation).
    Suppose we found a subsequence $\{\tilde{\gamma}_j\}$ weakly converging on $[0,N]$ for some $N\in\N$.
    Consider now the restriction $\{\tilde{\gamma}_j|_{[0,N+1]}\}$.
    This is uniformly bounded in $W^{2,2}(0,N+1;\R^2)$ and so has a (further) subsequence that converges weakly.
    Taking the diagonal sequence, we find the desired limit $\tilde{\gamma}\in\SC$.
    Lower semicontinuity of $\tilde{\ef{E}}^\lambda$ implies that for any $N$, the restriction
    $$\tilde{\gamma}^N := \tilde{\gamma}|_{[0,N]}$$
    satisfies
    $$\tilde{\ef{E}}^\lambda[\tilde{\gamma}^N] \leq \liminf_{j\to\infty}\tilde{\ef{E}}^\lambda[\tilde{\gamma}_j|_{[0,N]}] \leq \liminf_{j\to\infty}\tilde{\ef{E}}^\lambda[\tilde{\gamma}_j] = \inf_{\SC}\tilde{\ef{E}}^\lambda.$$
    Since $\lim_{N\to\infty}\tilde{\ef{E}}^\lambda[\tilde{\gamma}^N]=\tilde{\ef{E}}^\lambda[\tilde{\gamma}]$ by the monotone convergence theorem, we find that $\tilde{\gamma}$ is a minimiser of $\tilde{\ef{E}}^\lambda$ in $\SC$.

    Let $\tilde\gamma$ be any minimiser.
    Then the first variation \eqref{eq:first_variation} vanishes at least
    for all smooth perturbations $\eta:[0,\infty)\to\R^2$ with compact support such that $(\eta(0),e_1) = (\eta'(0),e_2) = 0$.
    Taking an appropriate family of functions $\eta\in C^\infty_c(0,\infty;\R^2)$ with bounded support shows that for every $[0,N]\subset\subset[0,\infty)$ the restriction
    $\tilde{\gamma}^N$ 
    is critical for $\tilde{\ef{E}}^\lambda$.
    Calculating, for any $\xi:[0,N]\to\R^2$,
    \begin{align*}
    \tilde{\ef{E}}^\lambda[\xi] &= \ef{B}[\xi]+\lambda\ef{L}[\xi] + \lambda\int_0^{N} \partial_s(\xi,e_1)\,ds
    \\
    &= \ef{B}[\xi]+\lambda\ef{L}[\xi] + \lambda(\xi(N) - \xi(0),e_1),
    \end{align*}
    which means that the first variation for $\tilde{\ef{E}}^\lambda$ with respect to perturbations of bounded support is equal to that for $\ef{B}+\lambda\ef{L}$.
    Thus $\tilde{\gamma}^N$ is also critical for $\ef{B}+\lambda\ef{L}$.
    Since $N$ is arbitrary, the (infinite-length) curve $\tilde{\gamma}$ must be a $\lambda$-elastica, and it must fit into the classification in Theorem \ref{thm:planar_explicit}.

    First, we note that a similarity image of case (III) is a possible candidate for $\tilde{\gamma}$.
    All the other cases can be ruled out by the uniform energy boundedness $\tilde{\ef{E}}^\lambda[\tilde{\gamma}]<\infty$.
    Indeed, case (I) is clearly excluded as the second term of $\tilde{\ef{E}}^\lambda$ is unbounded for any straight line in $\SC$. 
    All the remaining cases (II), (IV), (V) have nonzero periodic curvature and thus have unbounded bending energy as well.

    Thus $\tilde{\gamma}$ must be a similarity image of the borderline elastica $\gamma_b$.
    Since $\gamma_b$ is a $2$-elastica (see Remark \ref{rem:multiplier}) the scaling factor $\Lambda>0$ needs to satisfy $\lambda=2\Lambda^{-2}$ (see Lemma \ref{lem:scaling}) and hence for some isometry $M$,
    \[
    \tilde{\gamma}(s) = \sqrt{\tfrac{2}{\lambda}}M \gamma_b\Big(\sqrt{\tfrac{\lambda}{2}}(s+s_0)\Big)\,.
    \]
    By using the boundary condition of $\SC$ at $s=0$ and finiteness of the second term of $\tilde{\ef{E}}^\lambda[\tilde{\gamma}]$ (as a boundary condition at $s=\infty$) we deduce that $M\in\SM$ and $s_0=0$.
\end{proof}

This settles the borderline case not only for the length-fixed but also the length-variable problem.

\begin{corollary}\label{cor:unique_borderline}
    Let $\lambda>0$ and $\gamma$ be any $\lambda$-elastica given by a (finite length) subarc of a similarity image of the half-period borderline elastica $\gamma_b|_{[0,\infty)}$.
    Then $\gamma$ is the unique minimiser of $\ef{E}^\lambda$ in $\widetilde{\mathcal{A}}(\gamma)$.
    
    In particular, $\gamma$ is also the unique minimiser of $\ef{B}$ in $\mathcal{A}(\gamma)$.
\end{corollary}

\begin{proof}
    By almost the same argument for Corollary \ref{cor:unique_wavelike} we deduce that the subarc $\gamma$ is uniquely minimal for $\tilde{\ef{E}}^\lambda$ in the length-variable admissible set $\widetilde{\mathcal{A}}(\gamma)$.
    This minimality is equivalent to that for $\ef{E}^\lambda=\ef{B}+\lambda\ef{E}$ as already observed in the proof of Theorem \ref{thm:uniqminfbpborderline}.
    Hence the first assertion follows, and also the last assertion follows directly since $\lambda>0$.
\end{proof}

\begin{remark}
    Almost the same energy as our adapted energy was used in previous work \cite{miu20} for a related but different purpose.
    In that work it was represented in terms of the tangential angle $\theta$, namely in the form of a periodic potential perturbed by the Dirichlet energy $\varepsilon^2\int|\theta_s|^2ds + \int(1-\cos\theta) ds$.
    This expression highlights the fact that the borderline elastica can be regarded as a kind of ``transition layer''.
\end{remark}

We are now in a position to complete the proof of our main theorem.

\begin{proof}[Proof of Theorem \ref{thm:main_fixed_length}]
    By Theorem \ref{thm:planar_explicit}, any elastica with non-constant monotone curvature satisfies the assumption of one of Corollaries \ref{cor:unique_wavelike}, \ref{cor:unique_orbitlike}, and \ref{cor:unique_borderline}.
    Hence the desired unique minimality follows.
\end{proof}

\section{Free boundary problems with variable length}

In this section we extend some of previous results to the length-penalised problem.
We first address the wavelike case.
Here a discontinuous transition of minimisers occurs at $\ell=0$, and nontrivial minimisers only appear for $\ell\leq0$.

\begin{theorem}
\label{thm:uniqminfbp}
    Let $\lambda>0$ and $\ell\in\R$ be fixed parameters.
    Then there exists a smooth minimiser $\gamma_{\lambda,\ell}$ of the modified bending energy $\ef{E}^\lambda:\SF_\ell^+\rightarrow(0,\infty)$.

    If $\ell>0$, then any minimiser is a horizontal line segment.
    
    If $\ell\leq0$, then there is a unique parameter $m\in[m_0,1)$ (see Definition \ref{def:supercritical}), depending only on $\ell\sqrt{\lambda}$, such that any minimiser is parametrised by
    \begin{equation}\label{eq:wavelike_formula}
        \gamma_{\lambda,\ell}(s)=\sqrt{\tfrac{4m-2}{\lambda}} M\gamma_w\Big(\sqrt{\tfrac{\lambda}{4m-2}}s,m\Big), \quad s\in \big[ 0,2K(m)\sqrt{\tfrac{4m-2}{\lambda}} \big],
    \end{equation}
    for some $M\in\SM$ (see Lemma \ref{lem:invariancesM}), where $\gamma_w$ is the wavelike elastica in Theorem \ref{thm:planar_explicit}.
    
    Moreover, the inverse of
    \[
    m \mapsto 2\big( 2E(m)-K(m) \big)\sqrt{4m-2}
    \]
    is a strictly decreasing bijective function $\tilde\phi:(-\infty,0]\to[m_0,1)$ such that $\tilde\phi(\ell\sqrt{\lambda}) = m$.
\end{theorem}

\begin{proof}
    The case $\ell>0$ is trivial so hereafter we only consider the case $\ell\leq0$.
    
    The existence of a minimiser follows by using the direct method as in the fixed-length case and so we do not repeat it.
    The only difference is that here we use the a-priori bound on the energy $\ef{E}^\lambda$ to uniformly control the length $\ef{L}$ along a minimising sequence, with the assumption $\lambda>0$, instead of having it fixed a-priori.
    
    Let $\gamma_{\lambda,\ell}$ be any minimiser of $\ef{E}^\lambda$ in $\SF_\ell^+$.
    Then $\gamma_{\lambda,\ell}$ is also a minimiser of $\ef{B}$ in the fixed-length subclass $\SA^L$ with
    $$L:=\ef{L}[\gamma_{\lambda,\ell}].$$
    By Theorem \ref{thm:uniqminfixedL} this minimiser must be given by formula \eqref{eq:wavelike_fixedL_parametrisation} with a parameter $m\in(0,1)$ such that \eqref{eqn:monotoneqLfixed} holds.
    On the other hand, $\gamma_{\lambda,\ell}$ is clearly a $\lambda$-elastica, while by Lemma \ref{lem:scaling} and Remark \ref{rem:multiplier} we can compute the multiplier of the RHS of \eqref{eq:wavelike_fixedL_parametrisation} to deduce that
    $\lambda = 2(2m-1)(\tfrac{L}{2K(m)})^{-2}.$ 
    Since $\lambda>0$, we have $m>\frac{1}{2}$ and hence
    $$L=2K(m)\sqrt{\tfrac{4m-2}{\lambda}}.$$
    Substituting this $L$ to \eqref{eq:wavelike_fixedL_parametrisation} and  \eqref{eqn:monotoneqLfixed}, respectively, we obtain \eqref{eq:wavelike_formula} and
    \begin{equation}\nonumber
        2\big( 2E(m)-K(m) \big) \sqrt{4m-2} = \ell\sqrt{\lambda}.
    \end{equation}
    As the RHS is non-positive, so is the LHS, and hence $m\in[m_0,1)$ (see Definition \ref{def:supercritical}).
    For $m\in[m_0,1)$ the map $m \mapsto 2E(m)-K(m)$ is strictly decreasing and has range $(-\infty,0]$.
    Since $m_0>\frac{1}{2}$, the map $m \mapsto 2( 2E(m)-K(m) ) \sqrt{4m-2}$ also has the same property. 
    So the inverse gives the desired map $\tilde{\phi}$, and in particular $m\in[m_0,1)$ is uniquely determined.
\end{proof}

As in the length-fixed problem we obtain the following unique minimality, which only covers critical and supercritical wavelike elasticae as so does Theorem \ref{thm:uniqminfbp}.

\begin{corollary}\label{cor:unique_wavelike_variable}
    Let $\lambda>0$, $m\in[m_0,1)$, and $\gamma$ be a $\lambda$-elastica given by a subarc of a similarity image of the half-period wavelike elastica $\gamma_w(\cdot,m)|_{[0,2K(m)]}$.
    Then $\gamma$ is the unique minimiser of $\ef{E}^\lambda$ in $\widetilde{\mathcal{A}}(\gamma)$.
\end{corollary}

Before moving on to the orbitlike problem we mention the following

\begin{remark}
    In the proof of Theorem \ref{thm:uniqminfbp} we only used the monotonicity of the function $f:m \mapsto 2( 2E(m)-K(m) ) \sqrt{4m-2}$ on $[m_0,1)$, but the behavior on $(\frac{1}{2},m_0]$ is also useful for other problems.
    Calculating, $f''<0$ on $(\frac{1}{2},1)$ and $f(\frac{1}{2})=f(m_0)=0$ so $f$ takes its maximum $M_*:=\max f\simeq 0.837$ at $m=m_*\simeq 0.628\in(\frac{1}{2},m_0)$.
    We can apply this fact, for example, to classify all critical points of the modified bending energy $\ef{E}^\lambda$ with $\lambda>0$ subject to the so-called pinned boundary condition.
    More precisely, consider the class of curves such that one endpoint is fixed at the origin and the other is at $\ell e_1$ where $\ell\geq0$.
    Then as in the fixed-length case, the curvature of any critical point must vanish at the endpoints.
    This implies that any critical point is either a trivial line segment, or a wavelike elastica given by an isometric image of 
    $$\gamma(s):=N\sqrt{\tfrac{4m-2}{\lambda}}\gamma_w\Big(N^{-1}\sqrt{\tfrac{\lambda}{4m-2}}(s-K(m)),m\Big),\quad s\in[0,2NK(m)],$$ 
    where $m\in(\frac{1}{2},1)$ and $N\in\N$ such that
    $$\Big| 2( 2E(m)-K(m) ) \sqrt{4m-2} \Big|=\ell\sqrt{\lambda}/N.$$
    Finding all pairs of $(m,N)$ is reduced to understanding the behavior of the above function $f$.
    More precisely,
    \begin{itemize}
        \item for $N<\ell\sqrt{\lambda}/M_*$ there is a unique $m\in(m_0,1)$;
        \item for $N=\ell\sqrt{\lambda}/M_*$ there are two solutions, $m=m_*$ and $m\in(m_0,1)$;
        \item for $N>\ell\sqrt{\lambda}/M_*$ there are three solutions, $m\in(\frac{1}{2},m_*)$, $m\in(m_*,m_0)$, and $m\in(m_0,1)$.
    \end{itemize}
    In particular, if we focus on a fixed ``$N$-mode'', then as $\ell\to+0$ the last two solutions ``converge'' to a figure-eight elastica ($m\to m_0$), while the first one ``shrinks'' but after rescaling converges to a free elastica ($m\to \frac{1}{2}$).
    A similar observation is already given in \cite[Section 2.3]{Linner1998} in terms of the limit $\lambda\to+0$.
\end{remark}

Now we turn to the orbitlike case.

\begin{theorem}
\label{thm:uniqminfbporbitlike}
    Let $\lambda>0$ and $\ell<0$ be fixed parameters.
    Then there exists a smooth minimiser $\gamma_{\lambda,\ell}$ of the modified bending energy $\ef{E}^\lambda:\SF_\ell^-\rightarrow(0,\infty)$.
    
    In addition, there is a unique parameter $m\in(0,1)$, depending only on $\ell\sqrt{\lambda}$, such that any minimiser is parametrised by
    \begin{equation}\label{eq:orbitlike_formula}
        \gamma_{\lambda,\ell}(s)=\sqrt{\tfrac{4-2m}{\lambda}} M\gamma_o\Big(\sqrt{\tfrac{\lambda}{4-2m}}s,m\Big), \quad s\in\big[0,K(m)\sqrt{\tfrac{4-2m}{\lambda}}\big],
    \end{equation}
    for some $M\in\SM$ (see Lemma \ref{lem:invariancesM}), where $\gamma_o$ is the orbitlike elastica in Theorem \ref{thm:planar_explicit}.

    Moreover, the inverse of
    \[
    m \mapsto \big( 2E(m)+(m-2)K(m) \big) \frac{\sqrt{4-2m}}{m}.
    \]
    is a strictly decreasing bijective function $\tilde\psi:(-\infty,0)\to(0,1)$ such that $\tilde\psi(\ell\sqrt{\lambda}) = m$.
\end{theorem}

\begin{proof}
    The existence of a minimiser again follows by using the direct method, so we only argue about uniqueness.

    Let $\gamma_{\lambda,\ell}$ be any minimiser of $\ef{E}^\lambda$ in $\SF_\ell^-$.
    Then $\gamma_{\lambda,\ell}$ is also a minimiser of $\ef{B}$ in the fixed-length subclass $\SB^L$ with $L:=\ef{L}[\gamma_{\lambda,\ell}]$.
    By Theorem \ref{thm:uniqminfixedL} this minimiser must be given by formula \eqref{eq:orbitlike_fixedL_parametrisation} with a parameter $m\in(0,1)$ such that \eqref{eqn:monotoneqLfixedorbitlike} holds.
    On the other hand, $\gamma_{\lambda,\ell}$ is clearly a $\lambda$-elastica, while by Lemma \ref{lem:scaling} and Remark \ref{rem:multiplier} we can compute the multiplier of the RHS of \eqref{eq:orbitlike_fixedL_parametrisation} to deduce that
    $\lambda = 2(2-m)(\tfrac{L}{K(m)})^{-2}.$ 
    Hence 
    $$L=K(m)\sqrt{\tfrac{4-2m}{\lambda}}.$$
    Substituting this $L$ to \eqref{eq:orbitlike_fixedL_parametrisation} and \eqref{eqn:monotoneqLfixedorbitlike}, respectively, we obtain \eqref{eq:orbitlike_formula} and
    \[
    \big(2E(m) + (m-2)K(m)\big)\frac{\sqrt{4-2m}}{m} = \ell\sqrt{\lambda}.
    \]
    Let $f$ denote the function on the LHS of the above equation.
    We calculate
    \begin{align}
     f' = -\frac{(m^2 - 8 m + 8) E - (4m^2 - 12 m + 8) K}{(1-m) m^2 \sqrt{4 - 2m}}.
    \label{EQmagic}
    \end{align}
    Since $E - \sqrt{1-m}\,K \ge 0$
    (to see this note that it is zero at $m=0$ and the derivative is positive, cf.\ proof of Theorem \ref{thm:uniqminfixedLorbitcase}) we find
    \begin{equation}\label{EQmagic2}
    \begin{split}
            (m^2-8m+8)E
    &- (4m^2-12m+8)K
    \\
    &\ge \Big( (m^2-8m+8)\sqrt{1-m}
        - (4m^2-12m+8)\Big)K
        .
    \end{split}
    \end{equation}
    Furthermore, since
    \[
    (m^2-8m+8)^2(1-m)
        - (4m^2-12m+8)^2
        = m^4(1-m) > 0,
    \]
    we find $(m^2-8m+8)\sqrt{1-m} > (4m^2-12m+8)$.
    Using this in \eqref{EQmagic} and \eqref{EQmagic2} yields $f'<0$.
    Thus noting $f'(0+)=0$ we conclude that $m$ is unique and the inverse of $f$ gives the desired $\tilde{\psi}$.
\end{proof}

\begin{remark}
    The case $\ell\geq0$ can also be treated as in Remark \ref{rem:orbit_ell_0_positive}.
\end{remark}

\begin{corollary}\label{cor:unique_orbitlike_variable}
    Let $\lambda>0$, $m\in(0,1)$, and $\gamma$ be a $\lambda$-elastica given by a subarc of a similarity image of the half-period orbitlike elastica $\gamma_o(\cdot,m)|_{[0,K(m)]}$.
    Then $\gamma$ is the unique minimiser of $\ef{E}^\lambda$ in $\widetilde{\mathcal{A}}(\gamma)$.
\end{corollary}

\begin{proof}[Proof of Theorem \ref{thm:main_length_variable}]
    The proof is almost same as that of Theorem \ref{thm:main_fixed_length}.
    Here we use Corollaries \ref{cor:unique_wavelike_variable}, \ref{cor:unique_orbitlike_variable}, and \ref{cor:unique_borderline}.
\end{proof}

\section{Straightening problem}

In this final section we apply Theorem \ref{thm:main_length_variable} to prove Theorem \ref{thm:straightening}.
Our argument follows the general line of \cite{miu20}, thus going through the corresponding length-penalised problem.
For this application it is important to first check the following lemma, which is almost a direct consequence of the singular limit analysis in \cite{miu20}.
We introduce the variable-length admissible set 
$$\mathcal{A}_{\theta_0,\theta_1,\ell}:=\bigcup_{L>\ell}\mathcal{A}_{\theta_0,\theta_1,\ell}^L$$ 
by using the fixed-length admissible set $\mathcal{A}_{\theta_0,\theta_1,\ell}^L$ defined in \eqref{eq:admissible_clamped}.
Also, just for compatibility with \cite{miu20} we define the notation 
$$\ef{E}_\varepsilon:=\varepsilon^2\ef{B}+\ef{L}.$$
Clearly, minimisation of $\ef{E}_\varepsilon$ is equivalent for that of $\ef{E}^\lambda$ with $\lambda=1/\varepsilon^2$.

\begin{lemma}\label{lem:straightening_supercritical}
    Let $\ell>0$ and $\theta_0,\theta_1\in(-\pi,\pi)$ with $\theta_0\theta_1>0$.
    Then there exists $\varepsilon^*>0$ with the following property: Let $\varepsilon\in(0,\varepsilon^*)$ and $\gamma$ be any minimiser of the energy $\ef{E}_\varepsilon$ in the set $\mathcal{A}_{\theta_0,\theta_1,\ell}$.
    Then $\gamma$ is a supercritical wavelike elastica with monotone curvature.
\end{lemma}

\begin{proof}
    By \cite[Theorem 2.4]{miu20} there is a (small) $\varepsilon^*>0$ such that for any $\varepsilon\in(0,\varepsilon^*)$, any minimiser $\gamma$ is an elastica with exactly one sign change in its curvature.
    By classification (Theorem \ref{thm:planar_explicit}), up to reparametrisation and similarity, the curve must be of the form $\gamma=\gamma_w(\cdot,m)|_{[a,b]}$ for some subinterval $[a,b]\subset(-K(m),3K(m))$ and some parameter $m\in(0,1)$.
    Since $\gamma$ has multiplier $1/\varepsilon^2>0$, we have $m\in(\frac{1}{2},1)$.
    
    Thanks to the boundary layer analysis \cite[Theorem 2.2 (1)]{miu20} which ensures a rescaled convergence to a suitable borderline elastica as $\varepsilon\to0$, taking smaller $\varepsilon^*$ if necessary, we may also assume that the absolute curvature $|k|$ of $\gamma$ is decreasing (resp.\ increasing) in a neighborhood of the endpoint $s=0$ (resp.\ $s=\ef{L}[\gamma]$).
    This implies that the above subinterval must be contained in $[a,b]\subset[0,2K(m)]$ since $\gamma_w$ has absolute curvature $2\sqrt{m}|\cn(s,m)|$, which is increasing on $(-K(m),0)$ and $(K(m),2K(m))$ as well as decreasing on $(0,K(m))$ and $(2K(m),3K(m))$ (see e.g.\ the derivative formula for $\cn$ in \cite[Proposition A.23]{Miura2024survey}).
    Therefore, $\gamma$ is a wavelike elastica with monotone curvature.

    Finally we prove that $\gamma$ is supercritical in the sense of Definition \ref{def:supercritical}.
    Note that $\gamma$ has curvature of the form $2\sqrt{m}\Lambda^{-1}\cn(\Lambda^{-1}s+s_0,m)$, and as in the proof of formula \eqref{eq:orbitlike_formula} we can determine the scaling factor $\Lambda=\varepsilon\sqrt{4m-2}$.
    This function has period $4K(m)\Lambda=4\varepsilon K(m)\sqrt{4m-2}$.
    Since $\gamma$ is contained in a half period, meaning that $2\varepsilon K(m)\sqrt{4m-2}\geq \ef{L}[\gamma]$, and since the length has the bound $\ef{L}[\gamma]>\ell$, we have
    $$2\varepsilon K(m)\sqrt{4m-2}> \ell>0.$$
    Therefore, if $\varepsilon$ is sufficiently small, $K(m)$ needs to be so large that $m>m_0$.
\end{proof}

Now the key uniqueness is easy to verify.

\begin{corollary}\label{cor:straightening_uniqueness}
    Let $\ell>0$ and $\theta_0,\theta_1\in(-\pi,\pi)$ with $\theta_0\theta_1\neq0$.
    Then there exists $\varepsilon^*>0$ such that for any $\varepsilon\in(0,\varepsilon^*)$ the energy $\ef{E}_\varepsilon$ has a unique minimiser $\gamma_\varepsilon$ in the set $\mathcal{A}_{\theta_0,\theta_1,\ell}$.
\end{corollary}

\begin{proof}
    If $\theta_0\theta_1>0$, then this directly follows by Lemma \ref{lem:straightening_supercritical} combined with Theorem \ref{thm:main_length_variable}, while the case of $\theta_0\theta_1<0$ is already treated in \cite[Theorem 2.6]{miu20}.
\end{proof}

Armed with this uniqueness, we can prove Theorem \ref{thm:straightening} in the case of $\theta_0\theta_1>0$ in the completely same way as the proof in \cite[Section 6]{miu20} for the case of $\theta_0\theta_1<0$.
Here we give a refined argument for the control of the length of minimisers of $\ef{E}_\varepsilon$, and then apply it to complete the proof of Theorem \ref{thm:straightening}.

\begin{lemma}\label{lem:length_strict}
    Let $\ell\geq0$ and $\theta_0,\theta_1\in(-\pi,\pi]$ such that $\mathcal{A}_{\theta_0,\theta_1,\ell}$ does not admit a straight segment.
    Let $0<\varepsilon_0<\varepsilon_1$.
    For $j=0,1$, let $\gamma_j$ be any minimiser of $\ef{E}_{\varepsilon_j}$ in $\mathcal{A}_{\theta_0,\theta_1,\ell}$.
    Then $\ef{L}[\gamma_0]<\ef{L}[\gamma_1]$.
\end{lemma}

\begin{proof}
    We first recall that $\ef{L}[\gamma_0]\leq \ef{L}[\gamma_1]$ (as in \cite[Proposition 6.1]{miu20}), which follows by combining the two inequalities due to minimality:
    \begin{equation}\label{eq:length_strict1}
        \varepsilon_0^2\ef{B}[\gamma_0]+\ef{L}[\gamma_0]\leq\varepsilon_0^2\ef{B}[\gamma_1]+\ef{L}[\gamma_1], \quad \varepsilon_1^2\ef{B}[\gamma_1]+\ef{L}[\gamma_1]\leq\varepsilon_1^2\ef{B}[\gamma_0]+\ef{L}[\gamma_0].
    \end{equation}
    Now we prove that $\ef{L}[\gamma_0]\neq \ef{L}[\gamma_1]$ by contradiction.
    If $\ef{L}[\gamma_0]=\ef{L}[\gamma_1]$, then by \eqref{eq:length_strict1} we further have $\ef{B}[\gamma_0]=\ef{B}[\gamma_1]$, which implies that $\gamma_0$ (resp.\ $\gamma_1$) is also a minimiser of $E_{\varepsilon_1}$ (resp.\ $E_{\varepsilon_0}$) in $\mathcal{A}_{\theta_0,\theta_1,\ell}$.
    Then each of $\gamma_0$ and $\gamma_1$ is both a $(1/\varepsilon_0^2)$-elastica and a $(1/\varepsilon_1^2)$-elastica, but this is impossible since the curvature cannot vanish identically due to the boundary condition.
    We thus obtain a contradiction.
\end{proof}

\begin{corollary}\label{cor:length_behavior}
    Under the situation of Corollary \ref{cor:straightening_uniqueness}, the well-defined function $\varepsilon\mapsto \ef{L}[\gamma_\varepsilon]$ on $(0,\varepsilon^*)$ is strictly increasing and continuous.
    In addition,
    $$\lim_{\varepsilon\to0}\frac{\ef{L}[\gamma_\varepsilon]-\ell}{\varepsilon}=4\sqrt{2}\left(\sin^2\frac{\theta_0}{4}+\sin^2\frac{\theta_1}{4}\right).$$
\end{corollary}

\begin{proof}
    Uniqueness in Corollary \ref{cor:straightening_uniqueness} implies well-definedness directly, implies strict monotonicity combined with Lemma \ref{lem:length_strict}, implies continuity combined with \cite[Proposition 6.3]{miu20}, and also implies the last equation on the limit combined with \cite[Proposition 6.4]{miu20}.
\end{proof}

\begin{proof}[Proof of Theorem \ref{thm:straightening}]
    By Corollary \ref{cor:length_behavior}, taking the inverse map of $\varepsilon\mapsto \ef{L}[\gamma_\varepsilon]$, we can define a continuous, strictly increasing, bijective function $\tilde{\varepsilon}:(\ell,\ef{L}[\gamma_{\varepsilon^*}])\to(0,\varepsilon^*)$ with
    $$\lim_{L\to \ell}\frac{L-\ell}{\tilde{\varepsilon}(L)}=4\sqrt{2}\left(\sin^2\frac{\theta_0}{4}+\sin^2\frac{\theta_1}{4}\right).$$
    Then for any $L\in(\ell,\ef{L}[\gamma_{\varepsilon^*}])$, a curve $\gamma$ is a minimiser of $\ef{B}$ in $\mathcal{A}_{\theta_0,\theta_1,\ell}^L$ if and only if $\gamma$ is a minimiser of $E_{\tilde{\varepsilon}(L)}$ in $\mathcal{A}_{\theta_0,\theta_1,\ell}$.
    This directly gives a one-to-one correspondence between the fixed-length and length-penalised problems in the straightening regime under consideration.
    Therefore, up to a simple dilation argument as in \cite[Section 6.3]{miu20}, we can directly apply Corollary \ref{cor:straightening_uniqueness} to deduce the uniqueness in Theorem \ref{thm:straightening}, and also apply \cite[Theorems 2.2 and 2.4]{miu20} to deduce property (b) and the rescaled convergence in Theorem \ref{thm:straightening}.
\end{proof}

\bibliographystyle{plain}
\bibliography{monotonecurvature}

\end{document}